\documentclass{article}
\usepackage{amsthm}
\newtheorem{mydef}{Definition}[section]

\newtheorem{prop}{Proposition}[section]
\newtheorem{theo}{Theorem}[section]

\newtheorem{cor}{Corollary}[section]
\begin{document}

Tsemo Aristide,

College Bor\'eal,

1 Yonge Street, Toronto ON 

M5E 1E5, Canada.

tsemo58@yahoo.ca

\bigskip
\bigskip

{\bf Grothendieck topos, gerbes, and lifting actions of groups objects.}

\bigskip
\bigskip

\centerline{\bf Abstract.}

{\it  Let $C$ be a Grothendieck topos, $G$ and $H$ group objects of $C$. Let $p:P\rightarrow X$ be an $H$-torsor. Suppose that $X$ is endowed with an action of $G$. In this paper, we study the obstructions to lift the action of $G$ on $X$ to $P$ by using non commutative cohomology. Firstly, when a natural condition is satisfied,  we associate to this problem an extension of groups objects in $C$ whose splittings correspond to the liftings of the action of $G$. We apply the results obtained to the categories of topological and differentiable manifolds, and to the category of schemes. For the categories of differentiable manifolds and affine varieties defined over a closed field, we use also another approach induced by the slice theorems of Koszul and Luna which enable to define Grothendieck topologies for $G$-invariant neighborhoods.    This lifting problem has been studied in several categories by Brion, Hambleton, Hattori, Haussman, Lashof, May, Yoshida,... We recover and generalize  some of their results.}

\section{Introduction}

Let $M$ be a finite dimensional differentiable manifold endowed with the differentiable left action of the finite dimensional Lie group $G$. Consider a Lie group $H$ and a principal $H$-bundle $p:P\rightarrow M$. We say that the action of $G$ can be lifted to $P$, or equivalently that $P$ is  a $(G,H)$-bundle, if there exists  a left differentiable action of $G$ on $P$ which commutes with the right action of $H$ on $P$, and such that $p$ is a $G$-morphism.  $(G,H)$-bundles have been studied by many authors; among them, we can mention Brandt and Hausmann [3], Hambleton and Hausmann [11] and [12]. In [3], the authors have characterize the existence of liftings of the action of $G$ on $M$ to $P$ by using topological conditions and gauge theory. In [11] they have classified $(G,H)$-bundles for which the base space is $S^n$ endowed with the canonical action of $SO(n)$ which fixes the south and north poles and in [12] they have studied $(G,H)$-bundles over split spaces.

A similar question has been studied by many authors in the category of topological manifolds, among them, we can quote Hattori [13], Lashof, May, Segal [14], and Stewart [19];  their results are essentially homotopic.

In algebraic geometry, a similar question is also intensively studied in particular over algebraic homogeneous spaces. The purpose of this paper is to find a common foundation to this question which appears in three different fields of geometry. This foundation is the theory of topoi of Grothendieck. The concept of Grothendieck's topology was defined by Grothendieck in the setting of category theory which has leaded him to introduce the notion of  topos: this is a category equivalent to the category of sheaves defined on a $U$-category endowed with a Grothendieck's topology, where $U$ designs an universe. As said Grothendieck in r\'ecoltes et s\'emailles:

\smallskip
" Cette
id\'ee englobe, dans une intuition topologique commune, aussi bien les traditionnels espaces (topologiques),
incarnant le monde de la grandeur continue, que les (soi-disant) "espaces" (ou "vari\'et\'es") des g\'eometres
alg\'ebristes abstraits imp\'enitents, ainsi que d'innombrables autres types de structures, qui jusque la avaient
sembl\'e riv\'ees irr\'em\'ediablement au "monde arithm\'etique" des agr\'egats "discontinus" ou "discrets".
\smallskip

We start this paper by defining the lifting problem for a $H$-torsor $p:P\rightarrow X$ over an object $X$ of a topos $C$ endowed with the effective action of the group $G$. We suppose that $p:P\rightarrow X$ verifies the condition $C1$ which implies the existence of an exact sequence of group objects of $C$:

$$
1\rightarrow {\cal G}(P)\rightarrow Aut_G(P)\rightarrow G\rightarrow 1
$$

such that  the existence  of the lifting of the action of $G$ to $P$ is equivalent to the splitting of this exact sequence.   Such a question has been studied by Giraud in [6]. To study it, we apply the notion of gerbe and define a non commutative $2$-cocycle which is the obstruction of the existence of a lifting. This approach is an application  of the extension of Grothendieck classifying spaces studied by Giraud.

We apply this general theory to various situations. Firstly, we consider the lifting problem in the category of topological spaces. The results that we obtain here generalize the results of Hattori and Yoshida [13].  We also generalize some results obtained by Stewart [19] in this context.

 After, we consider the category of differentiable manifolds. We are particular interested to compact manifolds endowed with action of a compact Lie group. This enables to define two Grothendieck's topologies $b_{G,M}$ and $s_{G,M}$ adapted to the lifting problem by considering invariant neighborhood defined by Koszul. The first one is equivalent to a site defined on the principal blow-up of the $G$-manifod $M$. The sheaves of categories $p_G^b:F_G^M\rightarrow b_{G,M}$ (resp. $p_G^s:F_G^s\rightarrow s_{G,M}$) which represent the geometric obstruction of the lifting problem defined on $b_{G,M}$ (resp. $s_{G,M}$) are not always gerbes. Thus, to apply gerbe theory in these settings, we suppose that a condition $C2$ is verified.
When the  condition $C2$ is satisfied, we define gerbes which are subsheaves of categories of $p_G^b:F_G^b\rightarrow b_{G,M}$  which represent the geometric obstructions of the existence of a lifting of the action of $G$ to $P$ if the action of $G$ is principal.
To study the related question for $s_{G,M}$,  We add a condition $C3$ which expresses a relation between the isotropic groups. This enables us to classify $(G,H)$-bundles over $S^n$ endowed with the action of $SO(n)$ defined above.

By using the Luna's slice theorem, we also develop a theory similar to the theory developed in the context of differential geometry in algebraic geometry.

\section{Grothendieck topologies and torsors.}

The categories used in this section are ${\cal U}$-categories in respect of a fixed Grothendieck universe ${\cal U}$ which contains an element of infinite cardinal. {\bf A small category} will refer to a category which is ${\cal U}$-small. We also suppose that  the  categories used in this section are complete and cocomplete.  If $C$ is such a category, we denote by $1_C$ the final object of $C$ and for every object $X$ of $C$, $1_X:X\rightarrow 1_C$ is the final map.
We will adopt the following notations:

For every objects $X$ and $Y$ of $C$, we will often denote $Hom_C(X,Y)$ by $Y(X)$.

Let $(u_i:X_i\rightarrow X)_{i\in I}$ be a family of morphisms of $C$.
For every finite subset $\{i_1,...,i_p\}$  of $I$, we  denote by $X_{i_1...i_p}$ the fibre product $X_{i_1}\times_X...\times_XX_{i_p}$, by $u_{i_1...i_p}:X_{i_1...i_p}\rightarrow X$ and $u^{j_1...j_l}_{i_1...i_p}: X_{i_1..i_pi_{j_1}..i_{j_l}}\rightarrow X_{i_1...i_p}$ the canonical projections, where $\{j_1,...,j_l\}$ is another subset of $I$.

\begin{mydef}

Let $C$ be a category for every object $X$ of $C$, we denote by $C/X$ the {\bf comma category} over $X$. The objects of $C/X$ are morphisms $f:Y\rightarrow X$. A morphism $h$ between the objects $(f:Y\rightarrow X)$ and $(g:Z\rightarrow X)$ of $C/X$ is a morphism of $C:h:Y\rightarrow Z$ such that $g\circ h=f$.

A  full subcategory $S$  of $C$ is a {\bf sieve} if and only if for every morphism $f:X\rightarrow Y$ of $C$, the fact that $Y$ is an object of $S$ implies that $X$ is an object of $S$.

We say that the sieve $R$ is {\bf generated} by the family $(X_i)_{i\in I}$ of objects of $C$   if and only if $R$ is a subcategory of every sieve $R'$ of $C$ such that for every $i\in I$, $X_i$ is an object of $R'$.
\end{mydef}

\begin{mydef}

We say that the category $C$ is endowed with the {\bf Grothendieck  topology} $J$ if and only if
for every object $X$ of $C$, there exists a  non empty family of sieves $J(X)$ of $C/X$ such that:

1. For every element $R\in J(X)$ and every morphism $f:Y\rightarrow X$ of $C$, $R^f=\{g:Z\rightarrow Y: f\circ g\in R\}$ is an element of $J(Y)$.

2.  A sieve $R$ of $C/X$ is an element of $J(X)$ if and only if for every $R'\in J(X)$ and every object $f:Y\rightarrow X$ of $R'$, ${R}^f\in J(Y)$.

A {\bf site} $(C,J)$ is a category $C$ endowed with the Grothendieck topology $J$.

\end{mydef}

Remark that for every object $X$ of the site $(C,J)$, $J$ induces on $C/X$ a Grothendieck topology that we denote by $(C/X,J_X)$.

\begin{mydef}
Let $(C,J)$ be a site, $C^0$ the opposite category of $C$ and $Set$ the category of sets. A {\bf presheaf} defined on $C$ is a functor $S:C^0\rightarrow Set$. A {\bf sheaf} defined on $(C,J)$ is a presheaf $S$ such that:

for every object $X$ of $C$ and every sieve $R\in J(X)$, $lim_{Y\rightarrow X\in R}S(Y) =S(X)$.

For every morphism $f:X\rightarrow Y$ of $C$, the application $r(f):S(Y)\rightarrow S(X)$ is called the {\bf restriction} associated to $f$.

We denote  the category of sheaves defined on the site $(C,J)$ by $Sh(J,C)$.
\end{mydef}

The Yoneda embedding allows to define for each object $X$ of $C$ the presheaf $h_X$ on $C$ defined by $h_X(Y)=X(Y)$. {\bf The canonical topology of $C$} is the finest topology such that for each object $X$ of $C$, $h_X$ is a sheaf. Suppose that $C$ is endowed with its canonical topology $J$, the Yoneda embedding allows to identify $C$ to a subcategory of $Sh(C,J)$.

In the rest of the paper, we are going to suppose that the Grothendieck topology  $J$ of a site $(C,J)$ is less fine than the canonical topology of $C$. 

\begin{mydef}
A category $C$ is a topos if and only if $C$ is equivalent to the category of sheaves defined on a  $U$-site.

Let $(C,J)$ be a site for every object $X$ of $C$, {\bf the topos associated to $X$} $Top(X)$ is the topos equivalent to the category of sheaves defined on $(C/X,J_X)$.
\end{mydef}
 
 In this paper, we are going to study examples of algebraic  structures defined  in a category which can be defined intrinsically or with the Yoneda embedding as mentioned Grothendieck in [9]. We choose the second approach here to define the notions of group and group object in a category. See also Giraud [6] p. 106.
 
\begin{mydef}
A {\bf group object} $G$ in the category $C$ is an object $G$ such that for every object $X$ of $C$, the set $G(X)$ is endowed with the structure of a group  $(G(X),m_X)$ such that for every morphism $f:X\rightarrow Y$, the map $G(Y)\rightarrow G(X)$ which sends $g$ to $g\circ f$ is a morphism of groups.

A morphism $f$ between the group objects $G$ and $H$ is an element of $H(G)$ such that for every object $X$ of $C$, the morphism $G(X)\rightarrow H(X)$ which associates $f\circ g$ to every element $g$ of $G(X)$ is a morphism of groups. 
\end{mydef}

\begin{mydef}
Let $G$ be a group object of the category $C$, a {\bf subgroup} $L$ of  $G$ is a group object $(L,m_L)$ such that there exists  a monomorphism $im_L:L\rightarrow G$ which is a morphism of groups objects. 
\end{mydef}

Let $C$ be a category and $G$ a group object of $C$. We say that the object $X$ of $C$ is endowed with a {\bf  left action} $l_G^X$ of $G$ (resp. $X$ is endowed with a {\bf right action} $r_G^X$ of $G$) or equivalently that $X$ is a left $G$-object (resp. a right $G$-object) if and only if for every object $Y$ of $C$, the set $X(Y)$ is endowed with a structure of a left $G(Y)$-set $(X(Y),l_G^X(Y))$ (resp. with the structure of a right $G(Y)$-set $(X(Y),r_G^X(Y))$ such that for every morphism $f:Y\rightarrow Z$ of $C$, the map $X(Z)\rightarrow X(Y)$ which associates to $g\circ f$ to $g\in X(Z)$ is a morphism of left $G(Y)$-sets (resp. right $G(Y)$-sets).

We say that the left action defined by $l_G^X$ (resp. the right action $r_G^X$) is {\bf trivial} if for every object $Y$ of $C$, $(X(Y),l_G^X(Y))$ (resp. $X(Y),r_G^X(Y)))$ is trivial. 

Remark that if the object $X$ is endowed with a left action of $G$, we can associate to it a right action of $G$ such that for every object $Y$ of $C$, the group $G(Y)$ acts on $X(Y)$ by ${l_G^X(Y)}^{-1}$.

A subgroup $L$ of $G$ endows $G$ with a left $L$-action (resp. an $L$-right action) defined for every object $X$ of $C$ by the canonical left action (resp. right action) of $L(X)$ on $G(X)$ induced by the group structure of $G(X)$. 

Let $X$ be an object endowed with a left action (resp. with a right action) of $G$ we say that the {\bf left quotient} (resp. the {\bf right quotient}) of $X$ by $G$ exists if the presheaf defined on $C$ which associates to every object $Y$ of $C$ the left quotient of $X(Y)$ by $G(Y)$ (resp. the right quotient of $X(Y)$ by $G(Y)$) is representable. We denote by $G/X$ (resp. $X/G$) an object which represents it.

Remark that $G/X$ (resp. $X/G$) always exist in the topos $Sh(C,J)$ and is often called an algebraic space. 

Let $L$ be a subgroup of the group object $G$. We suppose that the quotient $L/G$  (resp. $G/L$) always exists and we call it a {\bf left homogeneous space} (resp. a {\bf right homogeneous space}). 

\medskip

We say that the left action of the group object $G$ and the right action of the group object $H$ defined on the object $X$ of $C$ commute if and only if for every object $Y$ of $C$, the left action of $G(Y)$ on $X(Y)$ and the right action of $H(Y)$ on $X(Y)$ commute.

\medskip
We recall the following definitions (see [6]. p. 117 and 126)
\begin{mydef}
Let $X$ be an object  object of the topos $(C,J)$ and $H$ a group object of $C/X$.  The right $H$-object $p:P\rightarrow X$ of $C/X$ is a $H$-{\bf torsor} if and only if:

- $p$ is epimorphic,

-$u:P\times_XH\rightarrow P\times_XP$ such that for every object $Y$ of $C$,

$u(Y):P(Y)\times_{X(Y)}H(Y)\rightarrow P(Y)\times_{X(Y)}P(Y)$ defined by $u(Y)(p,h)= (p,ph)$ is an isomorphism. 

The morphism $p:P\rightarrow X$ is a $H$-torsor of the site $(C/X,J_X)$ if $p$ is a $H$-torsor of the topos $Sh(C/X,J_X)$.
\end{mydef}

Remark that if $H$ is a group object of the topos $(C,J)$, we can associate to $H$ the group object $H_X$ of $C/X$ defined by $p_X:H\times X\rightarrow X$ where $p_X$ is the projection on the second factor and define the notion of $H_X$-torsor. Suppose that $P$ is a right $H_X$-object, the morphism $p:P\rightarrow X$ is an $H_X$-torsor if and only if there exists a sieve $R$ of $J(X)$, such that for every object $f_i:X_i\rightarrow X$, the pullback $p_i$ of $p$ by $f_i$ is isomorphic to $ p_{X_i}:X_i\times H\rightarrow X_i$.
See also [5] definition IV. 5.1.5. 

 Let $f_i^j:X_i\times_XX_j\rightarrow X_i$ be the projection on the first factor. There exists an isomorphism $u_{ij}:{f_j^i}^*p_j\rightarrow {f_i^j}^*p_i$ such that for every $Y$, $u_{ij}(Y):X_i(Y)\times_{X(Y)}X_j(Y)\times H(Y)\rightarrow X_i(Y)\times_{X(Y)}X_j(Y)\times H(Y)$ is an automorphism defined by $u_{ij}(Y)(x,y)= (x,yv_{ij}(x))$ where $v_{ij}(x)$ is an element of $H(Y)$. Remark that the family $(u_{ij})_{i,j\in I}$ verifies the Chasles relation that is: $u_{ij}^ku_{jk}^i = u_{ik}^j$. Conversely, a family of morphisms $u_{ij}:{f_j^i}^*p_j\rightarrow {f_i^j}^*p_i$ which satisfies the Chasles relation defines an $H_X$-torsor over $X$.

\begin{prop}
The action of $H_X$ of the $H_X$-torsor $p:P\rightarrow X$ induces an action of $H$ on $P$.
\end{prop}

\begin{proof}
Let $Y$ be an object of $C$, for every element $f$ of $Hom_C(Y,P)$, $p\circ f$ is an
 object of $C/X$.
 Let $g$ be an element of $Hom_C(Y,H)$, $(p \circ f,g)$ is an element of $Hom_{C/X}(p\circ 
 f,H_X)$. 
 To define the action of $g$ on $f$, consider $(p\circ f).(p\circ f,g)$ defined by the right action of $Hom_{C/X}(p\circ f,H_X)$ on $Hom_{C/X}(p\circ f,p)$. It is an element of $Hom_{C/X}(p\circ f,p)$ thus, it is defined by  a morphism $h:Y\rightarrow P$. We set $g.f=h$. 
\end{proof}

We are now able to present the lifting problem that we are going to study in this paper.

\medskip

{\bf Question 1.}

 Let $G$ and $H$ be group objects of the site $(C,J)$ and  $X$ be an object of  $C$ endowed with the left action $(X,l_G^X)$ of the group object $G$. Let $p:P\rightarrow X$ be a $H_X$-torsor defined over $X$, is there an action of $G$ on $P$ which commutes with $H$ and such that $p$ is a $G$-morphism ?
 
\medskip

Remark that since $X$ is a $G$-object, $Hom_C(X,X)$ is endowed with a left action of $G(X)$. For every $g\in G(X)$, we denote $h_g = g.Id_X$. Consider the subset $Aut_G(P)'$ of $Hom_C(P,P)$ such that for every $h$ in $Aut_G(P)'$, there exists $g$ in $G(X)$ such that: 

$$
h_g\circ p =p\circ h \eqno(1)
$$
$Aut_G(P)'$ defines a presheaf on $C$ such that for every $Y$ in $C$,  $Aut_G(P)'(Y)$ is the set of automorphisms of $P(Y)$ such that for every $f\in Aut_G(P)'(Y)$, there exists $g\in Aut_G(P)'$ such that for every $x$ in $P(Y)$,  $f(x)=g\circ x$.  We suppose that the presheaf defined by $Aut_G(P)'$ is representable by a group object $Aut_G(P)$.

Since the action of $G$ on $X$ is effective, the relation $(1)$ defines a morphism of groups objects $\pi_G:Aut_G(P)\rightarrow G$. 

\medskip

We suppose that the $H_X$-torsor $p:P\rightarrow X$ satisfies the following condition:

\smallskip

$C1$ 

The morphism of groups $\pi_G$ is an epimorphism.

\medskip

 We thus have an exact sequence:

$$
1\rightarrow {\cal G}(P)\rightarrow Aut_G(P)\rightarrow G\rightarrow 1 \eqno(2)
$$

The existence of a lifting of the action of $G$ on $P$ is equivalent to the fact that the previous exact sequence splits. We are going to study this exact sequences by using the notion of gerbe that we present in the next section.

\section{Sheaves of categories and gerbes.}

We are going to present in this section the main tool used to study the lifting problem mentioned above: it is the notion of sheaf of categories which enables to construct global objects by gluing local objects.

\begin{mydef}

Let $p:F\rightarrow C$ be a functor, for every object $X$ of $C$, the {\bf fiber} $F_X$ of $X$ is the subcategory of $F$ such that $Y$ is an object of $F_X$ if and only if $p(Y)=X$ and a morphism $f:Y\rightarrow Z$ of $C$ is a morphism of $F_X$ if and only if $p(f)=Id_X$.

\end{mydef}

 Let $m:A\rightarrow B$ be a morphism of $F$, we write $X=p(A)$, $Y=p(B)$,  and $n=p(m)$.
For every object $A'$ in $F_{X}$, we denote by $Hom_n(A',B)$ the  subset of $Hom_F(A',B)$ such that for every element $g\in Hom_n(A',B)$, $p(g)=n$. Consider the map $c_m(A'):Hom_{F_X}(A',A)\rightarrow Hom_n(A',B)$ which send $h$ to $m\circ h$.

\begin{mydef} We say that $m$ is {\bf Cartesian} if and only if $c_m(A')$ is a bijection for every object $A'$ of $F_X$.

The functor $p$ is a {\bf fibred category} if and only if for every morphism $n:X\rightarrow Y$ of $C$, and for every object $B$ of $F_Y$, there exists a Cartesian morphism $m:A\rightarrow B$ of $F$ such that $p(m)=n$ and the composition of two Cartesian morphisms is a Cartesian morphism. We will say often that $A$ is the restriction of $B$ by $n$ and write $A=r(n)(B)$. 

We say that the fibred category $p:F\rightarrow C$ is {\bf fibred in groupoids} if and 
if for every object $X$ of $C$, $F_X$ is a groupoid or equivalently that the morphisms of $F_X$ are invertible.
\end{mydef}

\begin{mydef}
Let $p:E\rightarrow C$, $p':E'\rightarrow C'$ be two fibred categories and $f:C\rightarrow C'$ a functor. A functor $g:E\rightarrow E'$ is a {\bf Cartesian} functor above $f$ if
and only if the image of a Cartesian morphism by $g$ is a Cartesian morphism and the following square is commutative: 

$$
\matrix{E &{\buildrel{g}\over{\longrightarrow}}& E'\cr p\downarrow &&\downarrow p'\cr C & {\buildrel{f}\over{\longrightarrow}}& C'}
$$

We denote by $Cart_f(E,E')$ the set of Cartesian functors above $f$.
\end{mydef}

Suppose $C$ is endowed with the topology $J$, for every object $X$ of $C$ and every sieve $R$ of $J(X)$, there exists a forgetful functor $i_R:R\rightarrow C$
which sends the object $h:Y\rightarrow X$ to $X$ and the morphism $h:(Z\rightarrow X)\rightarrow (Y\rightarrow X)$ to $h$. If $R$ is the sieve $C/X$ we denote $i_R$ by $i_X$.

\begin{mydef}

Let $(C,J)$ be a site. A fibred category $p:F\rightarrow C$ is a {\bf $J$-sheaf of categories} if and only if for every object $X$ of $C$ and every element $R\in J(X)$ the restriction functor 

$$
i_R^X:Cart_{i_X}(C/X,F)\rightarrow Cart_{i_R}(R,F)
$$
is an equivalence of categories.
\end{mydef}
 
\begin{mydef}

Let $(C,J)$ be a site and $p:F\rightarrow C$ a sheaf of categories defined on $(C,J)$. We say that $p$ is {\bf locally trivial} if and only if for every object $X$ of $C$, there exists a sieve $R$ in $J(X)$ such that
for every object $f:Y\rightarrow X$ of $R$, the category $F_Y$ is a non empty connected groupoid. A sheaf of categories which is locally trivial is often called a {\bf gerbe}.

Suppose that $M$ is the final object of $C$, the gerbe is {\bf trivial} if and only if $F_M$ is not empty.
\end{mydef}

\begin{mydef}
Let $p:F\rightarrow C$ be a gerbe defined on the site $(C,J)$. Suppose  that there is a sheaf in groups $L$ defined on $(C,J)$ such that for every object $e$ of $F_X$, there exists an isomorphism $i_e:L(X)\rightarrow Aut_{Id_X}(e)$ between $L(X)$ and the group of automorphisms $Aut_{Id_X}(e)$ of $e$ above the identity of $X$  such that:

 for every morphism $f:X\rightarrow Y$, every Cartesian morphism $u:x\rightarrow y$ above $f$ and every element $g\in L(X)$, $i_y(g)\circ u= u\circ r(f)(i_y(g)) $. In particular, the elements of $L(X)$ commutes with morphisms of $F_X$. 
 
We say that $L$ is the {\bf band} of the gerbe $p:F\rightarrow C$ or equivalently that the gerbe is {\bf bounded by $L$}.

A morphism of gerbes is a morphism of fibred categories which commutes with their respective band. 

\end{mydef}

\subsection{ The classifying cocycle.}

Let  $L$ be a sheaf define on the Grothendieck site $(C,J)$, an interesting question is to classify the gerbes defined on $(C,J)$ bounded by $L$, on this purpose Giraud [6] has  associated to a gerbe a non commutative clasifying cocycle that we describe in the following lines.

Let $p:F\rightarrow C$ be a gerbe bounded by $L$ and $M$ the final object of $C$.  Consider an element $R$ of $J(M)$ generated by $(u_i:U_i\rightarrow M)_{i\in I}$. For every $i\in I$ choose an object $e_i$ of $F_{U_i}$. Let $e_i^j=r(u_i^j)(e_i)$ be a restriction of $e_i$ over $U_{ij}$. Since $F_{U_{ij}}$ is a connected groupoid, there exists an isomorphism
$h_{ij}:e_j^i\rightarrow e_i^j$ above the identity.

Let $e_i^{jk}=r(u_{ij}^k)(e_i^j)$ be a restriction of $e_i^j$ to $U_{ijk}$. Recall that $e_i^{jk}$ is defined by a Cartesian morphism $l_i^{jk}:e_i^{jk}\rightarrow e_i^j$. Since the morphism $ l_i^{jk}$ is Cartesian, $l_i^{jk}$ and $h_{ij}\circ l_j^{ik}$ are above $u_{ij}^k$, there exists a   morphism $h_{ij}^k:e_j^{ik}\rightarrow e_i^{jk}$ such that $l_i^{jk}\circ h_{ij}^k=h_{ij}\circ l_j^{ik}$. We write:

$$
c_{ijk}= h_{ki}^jh_{ij}^kh_{jk}^i
$$

Remark that $c_{ijk}$ is an automorphism of $e_k^{ij}$ thus an element of $L(X_{ijk})$. The family of morphism $(c_{ijk})_{i,j,k\in I}$ is the classifying cocycle associated to the gerbe.

The definition of the $2$-cocycle $(c_{ijk})_{i,j,k\in I}$ depends on the choice of the $e_i$ and of the  Cartesian morphisms. A different choice of these data defines another $2$-cocycle and we say that this $2$-cocycle is cohomologous to $(c_{ijk})_{i,j,k\in I}$. This induces an equivalence relation on the set of $2$-cocycles bounded by $L$.
We denote by $H^2(C,J,L)$ the set whose elements are equivalence classes for this relation. see also [4].

We can now present the classifying theorem of Giraud [6]. For this purpose, we denote by $Gerb(C,J,L)$ the gerbes defined on the site $(C,J)$ bounded by $L$  and by $IsoGerb(C,J,L)$ the set of isomorphism classes of elements of $Gerb(C,J,L)$.

\begin{theo}
 Let $(C,J)$ be a site  and $L$ a sheaf defined on $(C,J)$. The correspondence $IsoGerb(C,J,L)\rightarrow H^2(C,J,L)$
 which associates to each equivalent class of gerbe the cohomologous class of its classifying cocycles is a bijection.
 \end{theo}

 {\bf Remark.}
 
 If the sheaf $L$ is commutative, then $H^2(M,L)$ is the set of cohomology classes in the classical sense and a gerbe is trivial if and only if the cohomology class of its classifying cocycle is zero.
 
 \medskip
 
 Suppose that the cohomology class of the classifying cocycle of the gerbe vanishes, a good question is to classify the global sections of the gerbe. This is achieved by the following result (see Giraud [6]):
 
 \begin{theo}
 Suppose that the classifying cocycle of the gerbe $p:F\rightarrow C$ bounded by the sheaf $L$ vanishes then the isomorphic classes of its set of global sections is in bijection with $H^1(C,J,L)$.
 \end{theo} 
 
{\bf Examples.}

\medskip

We are going to apply these notions to a sheaf of categories defined over the classifying topos of a group object to study the lifting problem. We start by the following definition due to Grothendieck (see Giraud [6] p. 411):

\begin{mydef}
Let $C$ be a topos and $G$ a group object of $C$, the classifying space of $G$ that we denote by $B_G$ is the topos whose objects are left $G$-objects.
\end{mydef}

Remark that a morphism $u:G_1\rightarrow G_2$ induces a morphism of topoi $B_u=(B_u^*,{B_u}_*):B_{G_1}\rightarrow B_{G_2}$. The inverse image $B_u^*$ of $B_u$ associates to a $G_2$-object $(U,l_{G_2}^U)$ the $G_1$ object $(U, l_{G_2}^U\circ u)$. The right adjoint ${B_u}_*$ of $B_u^*$ is defined by ${B_u}_*(V)=Hom_{G_1}(G_2,V)$.

Let $1\rightarrow L {\buildrel{l}\over{\longrightarrow}} M{\buildrel{m}\over{\longrightarrow}} N\longrightarrow 1$ be an exact sequence of groups defined in the topos $C$, and
let $Out(L)$ be the quotient of the group of automorphisms $Aut(L)$ of $L$ by the group of inner automorphisms of $L$.  The previous exact sequence induces a morphism $\phi:N\rightarrow Out(L)$, we deduce that $Out(L)$ is an $N$-object and $(L,\phi)$ defines a sheaf $L_N$ on the topos $B_N$. (See Giraud [6]   p. 430 and p. 431).

We are going to adapt here the result of Giraud [6] 5.3.1.

Let $E_N$ be the $N$-object of $B_N$ obtained by the action of $N$ on itself by left translations. The right translations define on $E_N$ the structure of a torsor on the final object $e_{B_N}$ of $B_N$ (see Giraud [6] p. 412) or [8] SGA4.1 p. 374 exercise 5.9.)

Consider the category $F_N^M$  such that an object of $F_N^M$ is a  $L$-torsor $p_V:e_V\rightarrow V$ over an object $V$ of $B_N$ such that the quotient of $e_V$ by $L$ is the pullback of $E_N$ by the canonical morphism $i_V:V\rightarrow e_{B_N}$.

A morphism of $F_N^M$ is defined by objects $e_{V_i}$, $i=1,2$   of $F_N^M$ over $V_i$, $f:V_1\rightarrow V_2$ a morphism of $B_N$ and $g:e_{V_1}\rightarrow e_{V_2}$ a morphism
of $L$-torsors such that $f\circ p_{V_1} =p_{V_2}\circ g$.  

Let $p_N^M:F_N^M\rightarrow B_N$ be the functor defined on objects by $p_N^M(e_V)=V$ and on morphisms by $p_N^M(g)=f$.

This result is an adaptation of the remark 6.2.11 p. 437 of Giraud [6].

\begin{theo}
The functor  $p_N^M:F_N^M\rightarrow B_N$  is a gerbe bounded by $L_N$. This gerbe is trivial if and only if the extension $1\rightarrow L\rightarrow M\rightarrow N\rightarrow 1$ splits. In this case, the set of isomorphic classes of the splittings of this extension is in bijection with $H^1(B_N,L_N)$.
\end{theo}

\begin{proof}

Let $f:V\rightarrow V'$ be a morphism of $B_N$ and $g:e_{V'}\rightarrow V'$ an object of the fibre ${F_N^M}_{V'}$. We are going to show that the pullback $h:e_V\rightarrow e_{V'}$ of $g$ by $f$ is a Cartesian morphism. This is equivalent to saying that for every object $e'_V$ of ${F_N^M}_V$ the map $c_h:Hom_{Id_V}(e'_V,e_V)\rightarrow Hom_f(e'_V,e_{V'})$ which sends $u$ to $h\circ u$ is bijective. This last assertion is a consequence of the universal property of  pullbacks: let $l:e_V\rightarrow V$ be the torsor projection map and $u,v$ be elements of $Hom_V(e'_V,e_V)$ such that $h\circ u = h\circ v$. We have $g\circ (h\circ u) =(g\circ h)\circ v =(f\circ l)\circ v$. The universal property of the pullback implies the existence of a unique $w:e'_V\rightarrow e_V$ such that $h\circ w =h\circ v$ and $l\circ w=l\circ v$. This implies that $w=v$ and by exchanging the role of $u$ and $w$, we obtain that $w=u$, thus $c_h$ is injective.

We show now that $c_h$ is surjective. Let $l':e'_V\rightarrow V$ be an object of $F_N^M$ and $u:e'_V\rightarrow e_{V'}$ such that $g\circ u=f\circ l'$. The universal property of the fibre product implies the existence of a morphism $v:e'_V\rightarrow e_V$ such that
$h\circ v=u$. This implies that $p_N^M$ is a fibred category.

Let $f:U\rightarrow V$  be an object of  $B_N/V$, we have to show that for every sieve   $R$  of $J(f)$, the restriction functor  $i_{R}^{f}:Cart_{i_{f}}(B_N/V,p_N^M)\rightarrow Cart_{i_{R}}(R,p_N^M)$ is an equivalence of categories, a fact which is equivalent to saying that this functor is essentially surjective and fully faithful.

Firstly, we show that $i_{R}^{f}$ is essentially surjective.  Consider an object $e$ of $Cart_{i_{R}}(R,p_N^M)$ and let $(f_i:V_i\rightarrow V)_{i\in I}$ be a family of objects of $R$ which generate $R$. Denote $e(f_i)$ by $p_i:e_i\rightarrow V_i$.

  Since $p_N^M$ is a fibred category, there exists a Cartesian morphism
$l_i^j:e_i^j\rightarrow e_i$ above the morphism $f_{ij}:V_i\times_V V_j\rightarrow V_i$. Let $e_{ij}=e(f_i\circ f_{ij})$. (Remark that $f_i\circ f_{ij}$ is an object of $R$). Since $l_i^j$ is a Cartesian map, we deduce the existence of a morphism $m_i^j:e_{ij}\rightarrow e_i^j$ such that $l_i^j\circ m_i^j =e(f_{ij})$; ($f_{ij}$ is a morphism of $R$). Remark that $m_i^j$ is invertible since $F_N^M(f_i\circ f_{ij})$ is a groupoid. Thus we can write $n_{ij}=m_i^j\circ {m_j^i}^{-1}:e_j^i\rightarrow e_i^j$.

Let $n_{ij}^k$ be the pullback of $n_{ij}$ above $V_i\times_VV_j\times_VV_k$,
We have the Chasles relation $n_{ik}^j=n_{ij}^kn_{jk}^i$. We deduce the existence of a torsor $p_e:e\rightarrow V'$ obtained by gluing the family $(e_i)_{i\in I}$. Moreover, there exist canonical maps $h_i:V_i\rightarrow V'$ such that $p_i$ is the pullback of $p_e$ by $h_i$. This implies that $i_{R_b}^{f_b}$ is essentially surjective.

Now we show that $i_{R}^{f}$ is fully faithful.

Let $e,e'$ be elements of $Cart_{i_{f}}(B_N^M/V)$ and $g,g':e\rightarrow e'$  morphisms. Let   $e_i$ (resp. $e'_i$) be the restriction of $e$  to $f_i$ (resp. the restriction of $e'$ to $f_i$). Suppose that $i_{R_b}^{f^b}(g)=i_{R_b}^{f_b}(g')$, then the restriction of $g$ and $g'$ to $e_i$ are equal. This implies that $g=g'$. Let $h:i_{R}^{f}(e)\rightarrow i_{R}^{f}(e')$ be a morphism, $h$ is defined by morphisms $h_i:e_i\rightarrow e_i'$ such that
the restriction $h_i^j$ of $h_i$ to the pullback $e_i^j$ of $e_i$ to $V_i\times_V V_j$ coincide with the restriction $h_j^i$ of $h_j$ to $e_j^i=e_i^j$. This implies that there exists a morphism $g:e\rightarrow e'$ whose restriction to $e_i$ is $h_i$. Thus $i_{R}^{f}$ is fully faithful and henceforth an equivalence of categories.

The gerbe $p_N^M$ is trivial if and only if it has a global section. This is equivalent to saying that the fibre of the final object $e_{B_N}$ of $B_N$ is not empty. An object of ${F_N^M}_{e_{B_N}}$ is up to the torsion by an $L$-torsor a group isomorphic to $M$ endowed with an action of $N$. This action is a splitting   of the exact sequence $1\rightarrow L\rightarrow M\rightarrow N\rightarrow 1$.

\end{proof}

We apply this result to the lifting problem. We have seen that if $p:P\rightarrow X$ is a $G$-torsor which satisfies the condition $C1$, there exists an exact sequence:
$$
1\longrightarrow {\cal G}(P){\buildrel{l(p)}\over{\longrightarrow}} Aut_G(P){\buildrel{aut(p)}\over{\longrightarrow}} G\rightarrow 1
$$
The previous theorem implies that $p_G^{Aut_G(P)}:F_G^{Aut_G(P)}\rightarrow B_{G}$ is a gerbe bounded by $L_{{\cal G}(P)}$. We have:

\begin{theo}
The action of $G$ on $X$ can be lifted to $P$ if and only if the classifying cocycle $c_p$  of the gerbe $p_G^{Aut_G(P)}$ is trivial. In this case, the set of isomorphism classes of the liftings of $G$ is in bijection with $H^1(B_{G},L_{{\cal G}(P)})$.
\end{theo}

We study now the lifting problem of a $H$-torsor $p:P\rightarrow X$ where $X$ is the right homogeneous space $G/L$. We have the following proposition:

\begin{prop}
Let $p:P\rightarrow G/L$ be an $H$-principal torsor  which satisfies the condition $C1$ and such that ${\cal G}(P)$ is a central subgroup of $Aut_G(P)$ and $H$ is commutative, then ${\cal G}(P)$ is isomorphic to $H$.
\end{prop}

\begin{proof}
Let $Y$ be an object of $C$, and $z$ an element of $P(Y)$. Consider the morphism $f:{\cal G}(P)(Y)\rightarrow H(Y)$ defined by $f(g) =a_g$ where $g(z)=za_g$. We are going to show that $f$ is an isomorphism. We have $(gg')(z)=za_{gg'} = g(za_{g'}) =(gz)a_{g'}=za_ga_{g'}$.
This implies that $f$ is a morphism of groups. Suppose that $f(g)=1_H(Y)$. Remark that $Aut_G(P)(Y)$ acts transitively on $P(Y)$ since $aut(P)$ is an epimorphism. Let $z'$ be an element of $P(Y)$, we can write $z'=g'(z), g'\in Aut_G(P)(Y)$. We have $g(z') = g(g'(z))=g'(g(z))=g'(z)=z'$. This implies that $f$ is injective.  Let $h$ be any element of $H(Y)$. We define the isomorphism $g_h$ of ${\cal G}(P)(Y)$ by $g_h(z)=zh$.
\end{proof}

The previous proposition implies that if ${\cal G}(P)$ is a central subgroup of $Aut_G(P)$ and $H$ is commutative, we have an exact sequence:

$$
1\rightarrow H\rightarrow Aut_G(P)\rightarrow G\rightarrow 1.
$$
We have the following result: (Compare to [19] proposition 1.1).

\begin{prop}
Suppose that the $(G,H)$-bundle $p:P\rightarrow G/L$ has a lifting. Then $P$ is a right homogeneous space of $G\times H$. In particular if $L$ is trivial, then $P$ is isomorphic to $P\times H$.
\end{prop}

\begin{proof}
let $Y$ be an object of $C$, the group $G(Y)\times H(Y)$ acts on $P(Y)$ by $(g,h).z=gzh^{-1}$. This action is transitive. This implies the existence of a subgroup object $L'$ of $G\times H$  such that $P$ is isomorphic to $(G\times H)/L'$. We can identify $L'(Y)$ with the stabilizer of $z$. Let $h$ be an element of $L'(Y)$, write $h=(h_1,h_2)$ where $h_1\in G(Y)$ and $h_2\in H(Y)$. Consider the morphism $f:L'(Y)\rightarrow G(Y)$ defined by $f(h)=h_1$; $f$ is injective since the action of $H$ on $P$ is free. Remark that the image of $f$ is the stabilizer of $p(Y)(z)$ that we can suppose without restricting the generality to be $L(Y)$. We deduce the existence of a morphism $g:L(Y)\rightarrow H(Y)$ such that $L'(Y)=\{(h,g(h); h\in L(Y)\}$. In particular, if $L$ is trivial, $L'$ is also trivial and $P=G\times H$. 
\end{proof}

\begin{prop}
Suppose that the $(G,H)$-bundle $p:P\rightarrow G/L$ has a lifting and there exists a representation $\phi:L\rightarrow H$ such that $P(Y)$ is the quotient of $G(Y)\times H(Y)$ by $\{(l,\phi(l)), l\in L(Y)\}$  then the group $Aut_G^H(P)$ of $(G,H)$-automorphisms is   the commutator  of $\phi(L)$ in $H$.
\end{prop}

\begin{proof}
Let $f:G\rightarrow G/L$ be the projection morphism, $Y$ an object of $C$ and $h$ a $(G,H)$-automorphism of $p(Y)$, we denote by $f^*h=m$ the pullback of $h$ to $f^*p(Y)$ $m$ is a $(G,H)$-automorphism of the trivial bundle $G(Y)\times H(Y)$. For every element $(u,v)$ of $G(Y)\times H(Y)$, and every $m\in Aut_G^H(f^*P)(Y)$, we can write $m(Y)(u,v) = (u,va_m(u,v))$. 

Since $m(Y)$ commutes with the action of $G(Y)$, we deduce that $a_m(u,v)$ depends only of $v$. Since $m(Y)$ commutes with $H(Y)$, we deduce that $m(Y)$ is a gauge automorphism. We can write $m(Y)(u,v) = (u,a_mv)$. Consider the map
$a:Aut_G^H(f^*P)(Y)\rightarrow H(Y)$ which sends $m$ to $a_m$. We show now that the image of $a$ is contained in the commutator of $\phi(L)(Y)$ in $H(Y)$. Since $m$ is the pullback of a morphism of $P$, we deduce that $l.m(u,v) = m(l.(u,v))$. This is equivalent to saying that $(lu,\phi(l)a_mv) = (lu,a_m\phi(l)v)$. We deduce that $\phi(l)a_m =a_m\phi(l)$.
 Thus $a_m$ commutes with  $\phi(L)$.  

 Remark that $a$ is injective.  Let $c$ be any element of the commutator  of $\phi(L)(Y)$ in $H(Y)$, we define the $(G,H)$-automorphism $a_c$ by $a_c(u,v) = (u,cv)$. This show that $a$ is bijective.   If we endow $H(Y)$ with the product defined by $xy=y^{-1}x^{-1}$, then $a^{-1}$ is a morphism of groups. 

\end{proof}

\medskip

{\bf Remark.}

A $(G,H)$-torsor over an homogeneous space $p:P\rightarrow G/L$ is often called a principal homogeneous bundle. More generally, if $F$ is any $L$-object, we can define the homogeneous bundle $P\rightarrow G/L$ where $P$ is the quotient of $G\times F$ by the diagonal action of $L$. This defines a functor between the Grothendieck classifying space $BL$ and the category of homogeneous $L$-spaces.

 Homogeneous bundles are intensively studied in differentiable geometry  and in algebraic geometry. The theory of gerbes enables to classify homogeneous bundles structures whose underlying principal $H$-bundle $p:P\rightarrow G/L$  verifies the property $C1$: If the obstruction $c_2\in H^2(BG,L_H)$ of the existence of $(G,H)$-structure vanishes, then the set of homogeneous bundles defined on $p$ is in bijection with $H^1(BG,L_H)$.

\section{Lifting groups in topology.}

We are going to apply in this section the results obtained in the previous section to study the lifting question in the category of topological manifolds. The results that we are going to obtain generalize the results of Hattori and Yoshida [13].

Let $C_{Top}$ be the category of topological manifolds,
 a group object in $C_{Top}$ is a topological group. The topological manifolds considered here are locally contractible.
 
  Let $H$ be a  topological group  which acts on the left (resp. on the right) of the topological   manifold $M$, $x$ an element of $M$ and $h$ an element of $H$. We denote by $h.x$ (resp. $x.h$) the action of $h$ on $x$. 
 In topology, an $H_M$-torsor  is a {\bf locally trivial principal $H$-bundle} $p:P\rightarrow M$: it is defined by a topological manifold $P$ on which $H$ acts freely on the right  such that $P/H$  is homomeomorphic to $M$. There exists an open covering $(U_i)_{\in I}$ of $M$ such that $p^{-1}(U_i)$ is isomorphic to the trivial $H$-bundle $U_i\times H$, and $P$ is obtained by gluing the trivial $H$-bundles $U_i\times H$ with the {\bf transition functions}
 $$
 u_{ij}:U_i\cap U_j\times H\rightarrow U_i\cap U_j\times H
 $$
 $$
 u_{ij}(x,y) = (x,y.g_{ij}(x)).
 $$
 Where $g_{ij}:U_i\cap U_j\rightarrow H$ is a continuous map.
 
 A {\bf gauge } transformation  of the $H$-principal bundle $p:P\rightarrow M$ is an homeomorphism $f$ of $P$ over the identity of $M$ which commutes with the action of $H$. Thus, for every element $y$ of $P$, there exists an element $a_y$ of $H$ such that $f(y)=y.a_y$. Remark that $a_{yh} = h^{-1}a_yh$. We denote by ${\cal G}(P)$ the group of gauge transformations of $P$. It is the set of global sections of the principal $H$-bundle ${\cal P}\rightarrow M$ whose transition functions are defined by:
 $$
 (x,y)\rightarrow (x,y.Ad(g_{ij}^{-1}))
 $$
 where $Ad$ is the adjoint representation of $H$.

 \medskip
 
 We can formulate the lifting problem in this context: we say that the action of $G$ on $M$ can be lifted to $P$ if and only if $P$ is endowed with a left action of $G$ such that for every $x\in P, g\in G$ and $h\in H$ we have:
$$
g.(x.h)=(g.x).h
$$
$$
g.(p(x))=p(g.x)
$$

{\bf Example.}

Let $G$ and $H$ be topological groups, the action of  $G$  on $G\times H$ defined by
$$
g.(g',h) = (gg',h)
$$
where $g,g'\in G$ and $h\in H$ lifts the canonical action of $G$ on itself by left translations.

\medskip

Suppose that the condition $C1$ is satisfied, we have the exact sequence of topological groups:

$$
1\longrightarrow {\cal G}(P){\buildrel{l(p)}\over{\longrightarrow}} Aut_G(P){\buildrel{aut(p)}\over{\longrightarrow}} G\rightarrow 1
$$

. We suppose that $C_{Top}$ is endowed with the weaker topology $J_{Top}$ such that a covering family of the topological space $X$ is a family of local homeomorphisms $(f_i:X_i\rightarrow X)_{i\in I}$ such that $\bigcup_{i\in I}f_i(U_i)=X$. Giraud [6] p. 453 shows that  to study extensions of groups in $C_{Top}$, we can study group extensions in the topos associated to $C_{Top}$. Since the sheaves defined on the final object of $C_{Top}$ are trivial, we have the following result (see also Giraud theorem 8.4):

\begin{theo}
The classiying class $c_p$ of the gerbe $p_G^{Aut_G(P)}:F_G^{Aut_G(P)}\rightarrow B_G$ is the obstruction of the existence of a lifting of $G$ to $P$. If this classifying cocycle vanishes then the set of isomorphism classes of the liftings of $G$ to $P$ is in bijection with $H^1(B_G,L_{{\cal G}(P)})$.
\end{theo}

{\bf Remarks.}

Remark that the condition $C1$ is satisfied if for every element $g\in G$, the pullback of $p:P\rightarrow M$ by $g$ is isomorphic to $p$.

 Hattori and Yoshida [13] has obtained a similar result by supposing the existence of a pseudo-lifting when the group $H$ is commutative.

Suppose that $H$ is the $n$-dimensional torus $T^n$, then $L_{{\cal G}(P)}=H^1(M,Z^n)$. We deduce that $H^2(BG,L_{{\cal G}(P)})=0$ if the first betti number $b_1(M)$ of $M$ is zero. If $G$ is compact connected and simply connected, it is well-known that $\pi_1(G)=\pi_2(G)=1$. The homotopy sequence attached to the universal fibration $E_G\rightarrow B_G$ shows that $\pi_1(B_G)=\pi_2(B_G)=1$, this implies that  $H^2(B_G,L_{{\cal G}(P)})=0$ if $H=T^n$. We thus obtain the following result (compare with Brandt and Haussmann [3] corollary 1.5 and 1.7, Hattori and Yoshida [13] Corollary 3.6 and with Stewart [19] theorem 4.1).

\begin{cor}
Let $M$ be a topological manifold endowed with the action of the compact Lie group $G$ and $p:P\rightarrow M$ be an $H$-principal bundle. Suppose that $H$ is commutative and one of the following condition is verified:

- $G$ is compact connected and $1$-connected,

-$b_1(M)=0$

 then the action of $G$ on $M$ can be lifted to $P$.
\end{cor}

\section{Applications to differentiable geometry.}

In this section, we are going to work in the category $C_{Diff}$ of differentiable manifolds and use the well-known results which describe actions of compact Lie groups in this category.

\medskip
  
  Suppose that $M$ is a compact manifold  endowed with the left action of the compact Lie group $G$. The {\bf orbit}  of the element $x$ of $M$ is the set $\{g.x, g\in G\}$. We denote by $G_x$ the 
  {\bf stabilizer} of $x$, it is the subgroup of $G$ whose elements fix $x$. We denote  by $[G_x]$ the conjugacy class of $G_x$ in $G$. We say that the elements $x$ and $y$ of $M$ have the same {\bf orbit type} if $G_x$ is conjugated to $G_y$.  The set of conjugacy classes of stabilizers of elements of $M$ is finite and is partially ordered by inclusion, and we denote by $[G_0],[G_1]...,[G_p]$ its elements. It can be shown that this set has a minimal element that we suppose to be $[G_0].$ Let $M_i$ be the subset of $M$ such that the stabilizers of elements of $M_i$ are conjugated to $G_i$; it is a submanifold of $M$ and $M$ is the union of $(M_i)_{i\in \{0,...,p\}}$. The submanifold $M_0$ is dense and the orbits of its elements are called the {\bf principal} orbits. The action of $G$ is {\bf principal} if and only there exists only one orbit type.  We say that the element  $x$ of $M$ is {\bf singular} if the orbit of $x$ is not principal. A submanifold $N$ of $M$ is {\bf singular} if and only if every element of $N$ is singular.
  
 Suppose that $M_i$ is singular or equivalently that $G_i$ is not conjugated to $G_0$. We can define the {\bf blowing-up} of $M$ along $M_i$ as follows: We choose on $M$ a differentiable metric preserved by the elements of $G$. There exists $r>0$ such that the disc subbundle    $D^r(M_i)$ of the normal bundle $N(M_i)$ of $M$ is a tubular neighborhood of $M_i$ and its boundary $S^r(M_i)$ is a submanifold of $M$.  The blowing-up $\tilde{M_i}$ of $M_i$ is the gluing of $M-D^r(M_i)$ and $S^r(M_i)\times [-1,1]$ by identifying their boundaries by
 $$
 (x,u,t)\rightarrow exp_x(tu)
$$
  where $x$ is an element of $M_i$, $u$ an element of $N(M_i)_x$ whose norm is $r$ and $t=1$ or $-1$.
  
 Remark that $\tilde{M_i}$ is endowed with the action of $G$ such that:
 
- The set of conjugacy classes of the stabilizers of the elements of $\tilde{M_i}$ is $\{[G_0],...,[G_p]\}-\{[G_i]\}$
 
- There exists a surjective $G$-morphism $p_i:\tilde{M_i}\rightarrow M$ whose restriction to $\tilde{M_i}-p_i^{-1}(M_i)$ is a diffeomorphism onto $M-M_i$.

We can iterate this process and obtain a manifold $\tilde{M}$ endowed with an action of $G$ and a surjective morphism $\tilde{p}:\tilde{M}\rightarrow M$ such that:

- The action of $G$ on $\tilde{M}$ is principal and the conjugacy classes of the stabilizers of elements of $\tilde{M}$ is $[G_0]$.

The restriction of $\tilde{p}$ to $\tilde{M}-\tilde{p}^{-1}({M_1,...,M_p})$ is a diffeomorphism onto $M-\bigcup_{i=1}^{i=p}M_i$.
  
\bigskip

Since there is a good understanding of the  geometry of the actions of compact Lie groups on compact manifolds, we  start by studying the locally the question 1, or equivalently by classifying $(G,H)$-bundles defined on homogeneous spaces, after that we use gerbe theory to study the global situation.

 Firstly, we have to ensure that given an element $x$ of $M$,  $\rho_x^*p\rightarrow G/G_x$ the pullback of $p$ by the orbit map $\rho_x$  is endowed with the structure of a $(G,H)$-bundle. This leads to the following condition which appears in the work of Brandt and Haussmann [3]:

\smallskip
C2. We suppose that for every element $x$ of $M$, and $y$ in $p^{-1}(x)$, there exists a continuous map $c:G\rightarrow P$ such that $c(1_G)=y$ and $p(c(g))=g.x$. 
\smallskip

The global question leads to the following question:

\medskip

{\bf Question 2.}

 Suppose that the action of $G$ can be lifted to the pullback of $p$ over an orbit of $G$, is this action can be lifted to $P$ ?
 
 \medskip

As we have seen, the local study is equivalent to the classification of homogeneous principal bundles. 
We define in this section Grothendieck topologies which occur in differential geometry. The two last examples are related to the lifting problem studied in this paper.

Let $Diff^{im}$ be the category whose objects are finite dimensional manifolds and whose morphisms are immersions. We define on the comma category $Diff^{im}/M$, over the object $M$ of $Diff^{im}$, the topology   $J_M^{s}$ such that for every object $h_U:U\rightarrow M$ of $Diff^{im}/M$, an element $R$ of $J_M^s(h_U)$ is a family of morphisms $(h_{U_i}:U_i\rightarrow U)_{i\in I}$ of $Diff^{im}/U$ such that  $\bigcup_{i\in I}h_{U_i}(U_i)=U$.  This topology is called the small site.

\medskip

Let $Diff^{op}$ be the category whose objects finite dimensional manifolds and whose objects are open maps. We define on the comma category $Diff^{op}/M$, over the object $M$ of $Diff^{op}$, the topology  $J_M^b$ such that for every object $h_U:U\rightarrow M$ of $Diff^{op}/M$, an element $R$ of $J_M^b(h_U)$ is a family of objects $(h_{U_i}:U_i\rightarrow U)_{i\in I}$ of $Diff^{op}/U$ such that  $\bigcup_{i\in I}h_{U_i}(U_i)=U$.  This topology is called the big site. 
\medskip

The following  two topologies will be particularly relevant for this paper. 

\medskip

Let $G$ be a  Lie group, $Diff_G^{im}$ be the category whose objects are finite dimensional $G$-manifolds and whose morphisms are $G$-immersions. Let $M$ be an object of $Diff_G^{im}$. We define on the comma category $Diff_G^{im}/M$ that we denote also $s_{G,M}$   the topology $J_G^s$  such that for every object $h_U:U\rightarrow M$ (also denoted $(h_U,U)$) an element $R$ of $J_G^s(h_U,U)$ is a family of objects $(h_{U_i}:U_i\rightarrow U)_{i\in I}$ of $s_{G,M}/U$ such that  $\bigcup_{i\in I}h_{U_i}(U_i)=U$.  This topology is called the small $G$-site.

\medskip

Let  $b_{G,M}$ be the category whose objects couples $(h_U,U)$ where $h_U$ is an  open $G$-differentiable maps $h_U:U\rightarrow M$ such that:

$U$ is the blowing-up of the singular subspace $N_U$ of $h_U(U)$,

The differentiable map $h_U$ is a blowing-up projection map $h_U:U\rightarrow h_U(U)$. Remark that the restriction of $h_U$ to $U-{h_U}^{-1}(N_U)$ is a diffeomorphism onto $U-N_U$.  

 A morphism  between the objects $h_U:U\rightarrow M$ and $h_{U'}:U'\rightarrow M$ of $b_{G,M}$ is an open $G$-differentiable map $f:U\rightarrow U'$ such that $h_{U'}\circ f = h_U$.  We define on $b_{G,M}$ the topology $J_G^b$ such that for every object $h_U:U\rightarrow M$ of $b_{G,M}$, an element $R$ of $J_G^b(h_U)$ is a family of objects $(h_{U_i}:U_i\rightarrow U)_{i\in I}$ of $b_{G,M}/h_U$ such that  $\bigcup_{i\in I}h_{U_i}(U_i)=U$. 

{\bf Remark.}

Let $M$ be a manifold endowed with the action of the Lie group $G$, if the action is principal then the categories $b_{G,M}$ and $s_{G,M}$ are identical. In fact $b_{G,M}$ is $s_{G,r(M)}$ where $r(M)$ is the resolution of $M$.

\medskip

Let $M$ be a differentiable manifold endowed with the left action of the Lie group $G$, for every element $x$ of $M$, denote by $V_x$  the quotient of the tangent space $T_xM$ of $x$, by its subspace $W_x$ tangent to the orbit of $x$. Remark that the differential $dg$ of every element $g$ of $G_x$ induces an automorphism of $V_x$. We deduce that the group $G_x$ acts diagonally on the left of $G\times V_x$ by: 
$$g.(g',v)=(g'g^{-1},dg(v)).
$$ 
and the quotient of $G\times V_x$ by this action is a vector bundle $p_x:S_x\rightarrow G/G_x$ defined over $G/G_x$ whose typical fiber is $V_x$. We will often use the following slice theorem due to Koszul:

\begin{theo}

 Let $M$ be a compact differentiable manifold endowed with the left action of the compact Lie group $G$, there exists a neighborhood $U$ (called a Koszul neighborhood) of the zero section of $p_x$ and a $G$-morphism $U\rightarrow M$ which is a diffeomorphism of $U$ onto its image.
\end{theo}

\medskip

{\bf Remarks.}

The slice theorem shows the existence of objects of the category $s_{G,M}$   and for every object $h_U:U\rightarrow M$ of $s_{G,M}$, elements of $J^s_{G}(h_U)$ by providing a family Koszul neighborhoods $(U_i)_{i\in I}$ such that $\bigcup_{i\in I}U_i=U$. 

We can also obtained objects of $b_{G,M}$ by blowing-up objects of $s_{G,M}$ along singular submanifolds.

\medskip

We are going apply the theory of gerbe  to study the question 2.

Let $M$ be a differentiable manifold endowed with the action of a Lie group $G$, $H$ a Lie group and $p:P\rightarrow M$ a locally trivial principal $H$-bundle defined over $M$. 
 We denote by $F_G^b$ (resp. $F_G^s$) the category whose objects are principal $H$-bundles $p_{e_N}:e_N\rightarrow N$ and which satisfy the following conditions:
 
 - There exists an object $f:N\rightarrow M$  of $b_{G,M}$  (resp. $s_{G,N}$) such that
  $e_N$ is isomorphic to the pullback $f^*p$ of $p$ by $f$,
  
  - $e_N$ is endowed with an action of $G$ which commutes with $H$ and $p_{e_N}$ is a $G$-morphism.
  
   A morphism $f:e_N\rightarrow e_{N'}$ of $F_G^b$ is a morphism of $H$-bundles which commutes with the action of $G$. This implies the existence of a $G$-morphism $g:N\rightarrow N'$ such that the following square is commutative:

$$
\matrix{e_N &{\buildrel{f}\over{\longrightarrow}}& e_{N'}\cr p_N\downarrow && \downarrow p_{N'}\cr N &{\buildrel{g}\over{\longrightarrow}} & N'}
$$

For every every object $u:N\rightarrow M$ of $b_{G,M}$ (resp. $s_{G,M}$), we denote by $F_G^b(u)$ (resp. $F_G^s(u)$) the  subcategory of $F_G^b$ (resp. $F_G^s$) which are $H$-bundles isomorphic to $u^*p$ and whose morphisms   are morphisms of $F_G^b$ (resp. $F_G^s$) over the identity of $u$. 

We can define the functor $p_G^b:F_G^b\rightarrow b_{G,M}$ (resp. $p_G^s:F_G^s\rightarrow s_{G,M}$) such that for every object $e_N\rightarrow N$ of $F_G^b$ over the object $h_N:N\rightarrow M$ of $b_{G,M}$ (resp. $F_G^s$ over the object $N$ of $s_{G,M}$) and every morphism $f:e_N\rightarrow e_{N'}$ of $F_G^b$ (resp. $F_G^s$) over $g:N\rightarrow {N'}$, $p_G^b(e_N)= h_N$ and $p_G^b(f)=g$ (resp. $p_G^s(e_N)= N$ and $p_G^s(f)=g$).

\medskip 
 
{\bf Remark.}

 The category $F_G^b(f)$ is a groupoid: let $g:e_N\rightarrow e'_N$ be a morphism of $F_G^b(f)$, the restriction of $g$ on each fibre of $e_N$ is a diffeomorphism since $f$ commutes with $H$, thus the inverse of $f$ can be calculated fiberwise.
 
\begin{prop}

{\it The functor $p_G^b:F_G^b\rightarrow b_{G,M}$  is a fibred category.}
\end{prop}

\begin{proof}

Let $f:N\rightarrow M$, $f':N'\rightarrow M$ be objects of $b_{G,M}$, $u:N\rightarrow N'$ be a morphism of $b_{G,M}$ and $e_{N'}$ an object of $F_G^b(f')$. The pullback $u^*(e_{N'})=e_N$ of $e_{N'}$ is an object of $F_G^b(f)$ defined by the Cartesian square:

$$
\matrix{e_N &{\buildrel{u^*}\over{\longrightarrow}}& e_{N'}\cr p_{N}\downarrow &&\downarrow p_{N'}\cr N & {\buildrel{u}\over{\longrightarrow}}& N'}
$$

We are going to show that $u^*$ is a Cartesian morphism. Let $e'_N$ be another object of $F_G^b(f)$. We have to show that the morphism
$c_u:Hom_{Id_N}(e'_N,e_N)\rightarrow Hom_u(e'_N,e_{N'})$ defined by $c_u(h)=u^*\circ h$ is bijective. This results from the general properties of the fibre products. (See also theorem 3.3).

To complete the proof, we are going to show that the composition of two Cartesian morphisms is a Cartesian morphism. Let $f_i:N_i\rightarrow M, i=1,2,3$ be objects of  
$b_{G,M}$, $u_i:N_{i+1}\rightarrow N_i, i=1,2$ be morphisms of $b_{G,M}$. Consider the Cartesian morphisms $v_i:e_{i+1}\rightarrow e_i, i=1,2$ Cartesian morphisms above $u_i, i =1,2$. Since $(u_1\circ u_2)^* =u_2^*\circ u_1^*$ is Cartesian and $F_G^b(f_3)$ is a groupoid, we deduce the existence of an isomorphism $l:e_3\rightarrow u_2^*(u_1^*(e_1))$ such that $v_1\circ v_2 =u_2^*\circ u_1^*\circ l$. Since $u_2^*\circ u_1^*$ is Cartesian and $l$ is an isomorphism, we deduce that  $v_1\circ v_2$ is a Cartesian morphism.

\end{proof}

 We can show with the same methods used to prove proposition 4.1 the following result:

\begin{prop}

The functors $p_G^s:F_G^s\rightarrow s_{G,M}$ is a fibred categories.
\end{prop}

\begin{prop}
 The functor $p_G^b:F_G^b\rightarrow b_{G,M}$  is a sheaf of categories.
\end{prop}

\begin{proof}
Let $f_b:N\rightarrow M$  be an object of  $b_{G,M}$, we have to show that for every sieve   $R_b$   $J_G^b(f_b)$ the restriction functor  $i_{R_b}^{f_b}:Cart_{i_{f_b}}(b_{G,M}/f_b,p_G^b)\rightarrow Cart_{i_{R_b}}(R_b,p_G^b)$ is an equivalence of categories, a fact which is equivalent to saying that this functor is essentially surjective and fully faithful.

Firstly we show that $i_{R_b}^{f_b}$ is essentially surjective.  Consider an object $e$ of $Cart_{i_{R_b}}(R_b,p_G^b)$ and let $(f_i:N_i\rightarrow N)_{i\in I}$ be a family of objects of $R_b$ such that $\bigcup_{i\in }f_i(N_i)=N$. Denote $e(f_i)$ by $p_i:e_i\rightarrow N_i$.

  Since $p_G^b$ is a fibred category, there exists a Cartesian morphism
$l_i^j:e_i^j\rightarrow e_i$ above the morphism $f_{ij}:N_i\times_N N_j\rightarrow N_i$. Let $e_{ij}=e(f_i\circ f_{ij})$. Since $l_i^j$ is Cartesian map, we deduce the existence of a morphism $m_i^j:e_{ij}\rightarrow e_j^i$ such that $p_G^b(l_i^j\circ m_i^j) =f_{ij}$. Remark that $m_i^j$ is invertible since $F_G^b(f_i\circ f_{ij})$ is a groupoid. Thus we can write $n_{ij}=m_i^j\circ {m_j^i}^{-1}:e_j^i\rightarrow e_i^j$.

Let $n_{ij}^k$ be the pullback of $n_{ij}$ above $N_i\times_NN_j\times_NN_k$,
We have the Chasles relation $n_{ik}^j=n_{ij}^kn_{jk}^i$. Since $e_i$ is a manifold, we deduce the existence of a manifold $e$ obtained by gluing the family $(e_i)_{i\in I}$ endowed with actions of $G$ and $H$ which commute each other and such that the action of $H$ is free. The quotient $N'$ of $e$ by $H$ is a manifold $N'$ endowed with an action of $G$ such that the canonical projection $p_e:e\rightarrow N'$ is a $G$-map. Moreover, there exists canonical maps $h_i:N_i\rightarrow N'$ such that $p_i$ is the pullback of $p_e$ by $h_i$. This implies that $i_{R_b}^{f_b}$ is essentially surjective.

Now we show that $i_{R_b}^{f_b}$ is fully faithful.

Let $e,e'$ be elements of $Cart_{i_{f_b}}(b_{G,M}/f^b,p_G^b)$ and $g,g':e\rightarrow e'$  morphisms. Let $(f_i:N_i\rightarrow N)_{i\in I}$ be objects of $R$ such that $\bigcup_{i\in I}f_i(N_i)=N$,  $e_i$ the restriction of $e$  to $f_i$ and $e'_i$ the restriction of $e'$ to $f_i$. Suppose that $i_{R_b}^{f^b}(g)=i_{R_b}^{f_b}(g')$, then the restriction of $g$ and $g'$ to $e_i$ are equal. This implies that $g=g'$ since $\bigcup_{i\in I}f_i(N_i)=N$. Let $h:i_{R_b}^{f_b}(e)\rightarrow i_{R_b}^{f_b}(e')$ be a morphism, $h$ is defined by morphisms $h_i:e_i\rightarrow e_i'$ such that
the restriction $h_i^j$ of $h_i$ to the pullback $e_i^j$ of $e_i$ to $N_i\times_N N_j$ coincide with the restriction $h_j^i$ of $h_j$ to $e_j^i=e_i^j$. This implies that there exists a morphism $g:e\rightarrow e'$ whose restriction to $e_i$ is $h_i$. Thus $i_{R_b}^{f_b}$ is fully faithful and henceforth an equivalence of categories.
\end{proof}

We can also show the following proposition by using the methods used to prove the proposition 5.3.

\begin{theo}
 The functor $p_G^s$ is a  sheaf of categories.
\end{theo}
 
We want to apply the theory of gerbes  to the sheaves of categories $p_G^s:F_G^s\rightarrow s_{G,M}$ and  $p_G^b:F_G^b(M)\rightarrow b_{G,M}$. Remark that these sheaves of categories are not always gerbes as shows the following example:

Suppose that $M$ is the point, let $Z$ be the group of integers, and consider the bundle $Z^2\rightarrow M$. Let $G$ be a discrete Lie group which acts (trivially) on $M$, there can exist more than two actions of $G$ on $Z^2$ which are not equivalent, thus in this situation, $p_G^s$ is not locally connected.

Here is another example: Suppose that $M=S^n$. Then $M$ is the quotient of the orthogonal group $SO(n+1)$ by $SO(n)$. Consider the trivial bundle $p:P=S^n\times SO(n+1)\rightarrow S^n$.
We can lift the action of $SO(n+1)$ on $P$ by the following actions: Let $g\in SO(n+1)$ and $x\in S^n$. For every $h\in SO(n+1)$, we set $\rho_1(h)(x,g)=(h.x,g)$ and $\rho_2(h)(x,g)=(h.x,hg)$. The lifts $\rho_1$ and $\rho_2$ are not equivalent since their orbits do not have the same dimension. This shows that the lifting of the action of $G$ on $P$ is not always unique.

 We are going to add conditions which insure that the sheaves of categories $p_G^s$ and $p_G^b$  are gerbes.  We start by studying the local situation.

\begin{prop}
Suppose that $M$ is connected and the action of $G$ on $M$ is principal then for every elements $x,y$ of $M$,
the pullback $\rho_x^*p$ of $p$ by the orbit map $\rho_x$ is isomorphic to $\rho_y^*p.$
\end{prop} 

\begin{proof}
We can suppose that a Koszul neighborhood of $x$ is isomorphic to the product $(G/G_x)\times V$ (see Audin proof of proposition 2.2.1.) where $V$ is a contractible ball of a finite dimensional vector space and that $y$ is an element of this Koszul neighborhood since $M$ is connected. This implies that the existence of an homotopy between $\rho_x$ and $\rho_y$. Let $F:M\rightarrow BH$ be the classifying map of $p$. The classifying map of $\rho_x^*p$ is $F\circ\rho_x$ which is homotopic to $F\circ\rho_y$. We deduce that
$\rho_x^*p$ is isomorphic to $\rho_y^*p$.
\end{proof}

\medskip

 Given an $H$-principal bundle: $p:P\rightarrow M$ such that $M$ is endowed with the principal action of the compact Lie group $G$, the results of Stewart[19] p. 194-195 motivates  the following condition:

\medskip
$C2$

 There exists an element $x$ of $M$, such that:

- there exists a morphism $l_x:G_x\rightarrow H$ such that the pullback of $p$ by the orbit map
$\rho_x$ is the canonical $H$-bundle of $P_x=(G\times H)/\Delta_x\rightarrow G/G_x$ where $\Delta_x={(g,l_x(g)), g\in G_x}$. 

 We suppose that $P_x$ is endowed with the action of $G$ induced by the action of $G$ on $G\times H$ defined by $g.(g_1,h)=(gg_1,h)$.
\medskip

 We consider the full subcategory $F^{b,l_x}_{G,M}$ of $F^b_{G,M}$ such that each object $p_U:e_U\rightarrow U$ of $F^b_{G,M}$ is a $(G,H)$-bundle such that for every $y$ in $U$, $\rho^*_yp_U$ is isomorphic to $P_x$. We denote by $p_G^{b,l_x}$ the restriction of $p_G^b$ to $F^{b,l_x}_{G,M}$.

We define now the sheaf $l^b_{H,M}$ on $b_{G,M}$ to be the sheaf defined on $b_{G,M}$ such that for every object $f:N\rightarrow M$ of $b_{G,M}$, $l^b_{H,M}$ is the set of  differentiable functions from $N$ to $H$ such that for every $x\in M$, $f(x)$ commutes with  $\l_x(G_x)$ in $H$.

\medskip

We have the following result:

\begin{theo}
Suppose that the condition   $C2$ is verified by the $H$-bundle $p:P\rightarrow M$ and $M$ is connected, then $p^{b,l_x}_{G,M}:F^{b,l_x}_{G,M}\rightarrow b_{G,M}$ is a gerbe bounded by $l_{H,M}^b$. If the classifying cocycle of $p^{b,l_x}_{G,M}$ vanishes, then the isomorphism classes of the $(G,H)$-bundles over $M$ such that there exists a point $y$ of $M$ such that $\rho_y^*p$ is isomorphic to $P_x$ is in bijection with $H^1(b_{G,M},l^b_{H,M})$.
\end{theo}

\begin{proof}
In order to prove this result, we can replace $M$ by its principal resolution and thus suppose without restricting the generality that the action of $G$ on $M$ is principal and the stabilizers of the orbits of its elements are conjugated to the subgroup $L$ of $G$. It remains to prove that $p_G^b$ is locally connected.

Let $(U_i)_{i\in I}$ be a open covering of $M$ by Koszul neighborhoods such that
$U_i$ is the quotient of $G\times V_i$ by  $L$ where $V_i$ is a contractible ball of a finite dimensional vector space. Since the action of $G$ is principal, we can suppose that $U_i$ is $G/L\times V_i$ and we denote by $f_i:G/L\times V_i\rightarrow M$ the canonical embedding. Let $x$ be an element of $V_i$ and $f_x:G/L\times \{x\}\rightarrow G/L\times V_i$ the canonical embedding. Since the property $C2$ is verified by $p$, there exists a representation $l_x:L\rightarrow H$ such that the pullback of $p$ by $f_i\circ f_x$ is the right quotient of $G\times H$ by $\Delta_x=\{(g,l_x(g)), g\in L\}$. Consider the $H$-bundle
$m_i:P_i\rightarrow U_i$ where $P_i$ is the product of the total space of the pullback of $p$ by $f_i\circ f_x$ and $V_i$. The restrictions of $m_i$ on $G/L\times \{x\}$ and $(f_i\circ f_x)p^*$ coincide. Since $U_i$ retracts to $L/G\times\{x\}$, we deduce that $f_i^*p$ is isomorphic to $m_i$. This implies  that $F_{G,M}^{b,l_x}(f_i)$ is not empty and  connected.

The fact that the gerbe $F^{b,l_x}_{G,M}\rightarrow b_{G,M}$ is bounded by the sheaf $l_{H,M}^b$ follows from the proposition 3.3.
\end{proof}

\begin{cor}
Suppose that the action of $G$ on the connected manifold $M$ is free,  then if the condition $C2$ is verified, $F^{b}_{G,M}\rightarrow b_{G,M}$ is a gerbe, in this case  for every  element $y$ of $M$, $\rho_y^*p$ is the trivial $H$-bundle. 
\end{cor}

\begin{proof}

If the action of $G$ is free, and the condition $C2$ is verified, then for every element $y$ of $M$, $\rho_y^*p$ is isomorphic to $G\times H$ since the stabilizers are trivial groups. We deduce that $F_G^b=F_{G,M}^{l_x,b}$ where $l_x$ is the trivial representation. We can henceforth apply the previous theorem.
\end{proof}

We have remarked that if the action of $G$ on $M$ is principal, then $b_{G,M}$ is equal to $s_{G,M}$ we deduce:

\begin{cor}
Suppose that the action of $G$ on $M$ is principal, and  the condition $C2$ is verified by the $H$-bundle $p:P\rightarrow M$, then $p^{b,l_x}_{G,M}:F^{b,l_x}_{G,M}\rightarrow s_{G,M}$ is a gerbe bounded by $l_{H,M}^b$. 
\end{cor}

Lashof [14] has studied $(G,H)$-bundles for which the action of $G$ is principal by using homotopy theory.

\medskip

We are going to study now the lifting problem on the site $s_{G,M}$.
Let $G$ be a Lie group which acts on the compact manifold $M$. Recall that if $M$  is a disjoint union $\bigcup_{i=0}^{i=p}M_i$ where $M_i$ is a connected component of a submanifold  type if $i>0$ and  $M_0$ is the submanifold type of the principal orbits. We can assume without restricting the generality that for every $i>0$, there exists a neighborhood $U_i$ of $M$ invariant by $G$ and such that the intersection of $U_i$ and $U_j$ is empty if $i,j>1$. Let $x_i$ be an element of $M_i$ and $L_i$ its stabilizer. We suppose that $L_0$ is a subgroup of $L_i, i>0$.

We set now a new condition to define a classifying gerbe.

\medskip

Condition $C3$.

Suppose that the condition $C2$ is satisfied and there exists a representation
$h_i:L_i\rightarrow H$ such that:

For every $i,j>0$ if  $L_i\subset L_j$, then restriction of $h_j$ to $L_i$ is $h_i$.

$\rho_{x_i}^*p$ is isomorphic to the right quotient of $G\times H$ by $L_i$ defined by $l.(x,y) = (lx,h_i(l)y)$,

$\rho_{x_0}^*p$ is isomorphic to the right quotient of $G\times H$ by $L_0$ defined by $l.(x,y) = (lx,h_i(l)y)$.

We define,  $F^{h}_M$ the subcategory of $F^s_M$ such that for every object $e_U\rightarrow U$ of the  $F^h_M$,
and $x$  an element of $U$,

 if $x$ is in $M_i$, $i\geq 1$, then there exists a Koszul neighborhood, $U'\subset U$ of $x$ which  is  the quotient of $G\times V$ by $L_i$. We suppose that the restriction of $e_U$ to $U'$ is isomorphic to the quotient of $G\times V\times H$ by $L_i$ by the action defined by $l(x,y,z)=(lx,dl(y),h_i(l)z), i>1$.

We denote by $l^{h,s}_M$ the sheaf defined on $s_{G,M}$ such that for every object
$U$ of $s_{G,M}$, $l^{h,s}_M$ is the set of differentiable  functions from $U$ to $H$ such that for every $x\in U$, $f(x)$ commutes with  $l_x(G_x)$.
Let $p^{h}_G:F^h_M\rightarrow s_{G,M}$ be the restriction of $p^s_M$ to $F^h_M$. The condition $C3$ expresses the fact that the functor $p^h_G$ is a locally connected sheaf of categories, thus a gerbe. The proposition 3.3 implies that this gerbe is bounded by $l^{h,s}_M$. These facts are summarize by the following result:

\begin{theo}
Suppose that the conditions  $C2$ and $C3$ are verified, then the functor
$p^h_G:F^h_M\rightarrow s_{G,M}$ is a gerbe bounded by $l^{h,s}_M$. 
\end{theo}

We consider now the particular situation where $M=S^n$.  Consider the action of $SO(n)$ on $S^n\subset R^{n+1}$ which fixes the south pole $S$ and the north pole $N$.
This action of $SO(n)$ has two orbit types: one type is represented by the poles, the fixed points of this action. The other type is represented by the intersection of $S^n$ and the hyperplanes orthogonal to the line containing the north and the south poles. We have the following result: (compare with Hambleton and Haussmann [11]).

\begin{theo}
Suppose that $S^n$ is endowed with the action of $SO(n)$ that we have just defined and let $p:P\rightarrow S^n$ be a principal $H$-bundle. Suppose that the conditions $C2$ and $C3$ are verified, then the gerbe $p^h_G:F^h_M\rightarrow s_{G,S^n}$ is trivial. The set of isomorphic classes of $(SO(n),G)$-bundles over $S^n$  is in bijection with $H^1(S^n,l^{h,s}_M)$.
\end{theo}

\begin{proof}
Let $U_S$ be the south hemisphere of $S^n$ and $U_N$ its north hemisphere. The restriction of $p^h_G$ to $U_S$ and $U_N$ are trivial, since $U_N$ and $U_S$ are contractible. Let $p_S:e_S\rightarrow U_S$ and $p_N:e_N\rightarrow U_N$ be  objects of $F^h_G$. There are trivial bundles since $U_N$ and $U_S$ are contractible. Thus there exists representations $\rho_S,\rho_N:SO(n)\rightarrow H$ such that $e_S$ (resp. $e_N$) is the quotient of $U_S\times H$ (resp. $U_N\times H$) by action of $SO(n)$  defined by
$h(x,y) =(h.x,\rho_S(h)(y))$ (resp. $h(x,y)=(h.x,\rho_N(h)(y))$. Since the condition $C3$ is verified, we can glue $(U_S\times H)/SO(n)$ and $(U_N\times H)/SO(n)$ along $(S^{n-1}\times H)/SO(n)$ and obtain a global object of $p^h_G$. 
\end{proof}

\bigskip

{\bf Applications to algebraic geometry.}

\bigskip

Let $S$ be a scheme, and $(C_S,J_S)$  the category $C_S$ of $S$-schemes endowed with the \'etale   topology $J_S$. Consider the group schemes $G$ and $H$ of $C_S$ and  let $X$ be a $S$-scheme  endowed with an action of $G$ and $p:P\rightarrow X$ be an $H$-torsor. Can the action of $G$ be lifted to $P$ to an action which commutes with $H$ ?

\medskip
We can apply here
 the results obtained in the general situation of topoi. We also have have the exact sequence:
 $$
 1\rightarrow {\cal G}(P)\rightarrow Aut_G(P)\rightarrow G
 $$
Where ${\cal G}(P)$ is the group of gauge automorphisms of $p$ and $Aut_G(P)$ the group of automorphisms of $P$ above elements of $G$. Suppose that the condition $C1$ is verified, the theorem 3.4 implies the following result: 
 
\begin{theo}
Suppose that the condition $C1$ is verified; the action of $G$ on $X$ can be lifted to $P$ if and only if the classifying cocycle $c_p$  of the gerbe $p_G^{Aut_G(P)}:F_G^{Aut_G(P)}\rightarrow B_{G}$ is trivial. In this case, the set of isomorphism classes of the liftings of $G$ is in bijection with $H_{et}^1(B_{G},L_{\cal G})$.
\end{theo}

Let $X$ be a smooth affine irreducible variety defined over the algebraically closed field $k$, and $G$ a reductive group which acts on $X$. Let $x$ be an element of $X$ whose orbit is closed. Luna [16] has shown that:

\begin{theo}
There exists an affine subvariety $V$ of $X$ stable by $G_x$ which contains $x$ and an \'etale $G$-morphism $f:G\times_{G_x}V\rightarrow X$ whose image is an open 
affine subvariety of $X$.
\end{theo}

Luna [16] has also defined a stratification of $X$ as follows: 

We denote by ${\cal M}$ the space whose elements are homogeneous vectors bundles isomorphic to $G\times_HN$ where  $H$ is a reductive subgroup of $G$, and $N$ a finite dimensional $k$-vectors space endowed with an action of $H$.   Let $X/G$ be the category quotient of $X$ by $G$ $p_X:X\rightarrow X/G$ the projection morphism. We can define the map $\mu_X:X/G\rightarrow {\cal M}$ such that $\mu_X(a)$ is the class of the normal bundle of ${p_X}^{-1}(a)$. The image of $\mu_X$ is finite. This implies that there exists a finite number of orbit types that we denote by $[L_1],...,[L_n]$. Let $\lambda$ be an element of ${\cal M}$, $M_{\lambda}={p_X}^{-1}((\mu_X)^{-1}(\lambda))$ is a closed subvariety of $X$. Moreover there exists an open stratum$M_{\lambda_1}$ called the principal stratum.

This result enables to define the site $Et^G_X$ whose objects are \'etale morphisms $h:Y\rightarrow X$ such that $Y$ is endowed with an action of $G$, $h$ is a $G$-morphism and its image is open.

Consider the category $F_X^G$ whose objects are $H$ bundles $p_{e_U}:e_U\rightarrow U$ such that:

- There is an \'etale morphism $h_U:U\rightarrow X$ which is an object of $Et^G_X$,

- The bundle $p_{e_U}:e_U\rightarrow U$ is the pullback of $p:P\rightarrow X$ by $h_U$,

- $e_U$ is endowed with a $(G,H)$-structure such that $p_{e_U}$ is a $G$-morphism.

A morphism of $F_X^G$ is defined by a morphism  $f:e_U\rightarrow e_{U'}$ of $H$-torsors such that there exists a morphism $g:U\rightarrow U'$ of $Et_X^G$ such that $h_{U'}\circ f = g\circ h_U$.

We define $p_X^G:F_X^G\rightarrow Et_X^G$ to be the functor which sends $h_U$ to $U$ and the morphism $f:e_U\rightarrow e_{U'}$ above $g$ to $g$.

\medskip

We suppose that the following condition is satisfied:

$C3a$ There exist morphisms $h_i:L_i\rightarrow H$ such that:

If $L_i\subset L_j$ then $h_i$ is the restriction of $h_j$ to $L_i$.

Let $x$ be an element of $M_{\lambda_i}$. The pullback of $p:P\rightarrow X$ by the orbit map $\rho_x$ is isomorphic to the quotient of $G\times H$ by $L_i$.

We denote by $l_X^H$ the sheaf defined on $Et_X^G$ such that for every object $U$ of $Et_X^G$, $l_X^H(U)$ is the set of morphisms $U\rightarrow H$ which are constant on the orbits of $G$  and such that for every $x\in U$ and $f\in l_X^H(U)$, $f(x)$ is in the commutator of $L_i$ if $x$ is fixed by $L_i$.

\begin{theo}
Suppose that the condition  $C3a$ is satisfied then $p_X^G$ is a gerbe bounded by $l_X^H$.
\end{theo}

\bigskip
\bigskip

{\bf References.}

\bigskip
\bigskip

1. AUDIN, Michele. Op\'erations hamiltoniennes de tores sur les vari\'et\'es symplectiques:quelques m\'ethodes topologiques, 1989.
\smallskip

2. BRION, Michel. On automorphism groups of fiber bundles. arXiv preprint arXiv:1012.4606, 2010.

\smallskip
3. BRANDT, Didier and HAUSMANN, Jean-Claude. Th\'eorie de jauge et sym\'etries des fibr\'es. In : Annales de l'institut Fourier. 1993. p. 509-537.
\smallskip

4. BREEN, Lawrence et MESSING, William. Differential geometry of gerbes. Advances in Mathematics, 2005, vol. 198, no 2, p. 732-846.
\smallskip

5. DEMAZURE, Michel et GROTHENDIECK, Alexandre. Sch\'emas en groupes: Groupes de type multiplicatif et structure des sch\'emas en groupes g\'en\'eraux.  1970.
\smallskip

6. GIRAUD, Jean. Cohomologie non ab\'elienne. 1966
\smallskip

7. GIRAUD, Jean. M\'ethode de la descente. M\'emoires de la Soci\'et\'e Mathématique de France, 1964, vol. 2, p. III1-VIII150.
\smallskip

8. GROTHENDIECK, Alexandre. SGA 4: Th\'eorie des Topos et Cohomologie \'Etales des Sch\'emas, tome 1–3. Lecture Notes in Math, vol. 269.

\smallskip
9. GROTHENDIECK, Alexandre. Introduction au calcul fonctoriel. Alger 1967.

\smallskip
10. GROTHENDIECK, Alexandre, R\'ecoltes et s\'emailles.

\smallskip
11. HAMBLETON, Ian and HAUSMANN, Jean-Claude. Equivariant principal bundles over spheres and cohomogeneity one manifolds. Proceedings of the London Mathematical Society, 2003, vol. 86, no 01, p. 250-272.

\smallskip
12. HAMBLETON, Ian and HAUSMANN, Jean-Claude. Equivariant bundles and isotropy
 representations. arXiv preprint arXiv:0704.2763, 2007.
 
 \smallskip
13. HATTORI, Akio and YOSHIDA, Tomoyoshi. Lifting compact group actions in fiber bundles. Japan. J. Math, 1976, vol. 2, no 1, p. 13-25. 

\smallskip
14. LASHOF, Richard, et al. Equivariant bundles over a single orbit type. Illinois Journal of Mathematics, 1984, vol. 28, no 1, p. 34-42.

\smallskip
15. LASHOF, R. K., MAY, J. P., et SEGAL, G. B. Equivariant bundles with abelian structural group. In : Proceedings of the Northwestern Homotopy Theory Conference (Evanston, Ill., 1982). 1983. p. 167-176.

\smallskip
16. LUNA, Domingo. Slices \'etales. M\'emoires de la Soci\'et\'e Math\'ematique de France, 1973, vol. 33, p. 81-105.

\smallskip
17. NGUYEN Dinh Ngoc, Cohomologie non ab\'elienne et classes caract\'eristiques. C. R. Ac. Sc. 251, 2453
(1960).

 \smallskip
18. SERRE, Jean-Pierre, Cohomologie Galoisienne, 1965.

\smallskip
19. STEWART, T. E. Lifting group actions in fibre bundles. Annals of Mathematics, 1961, p. 192-198.
\smallskip

20. TSEMO,  Aristide. Non abelian cohomology: the point of view of gerbed tower. Afr. Diaspora J. Math. 4 (2007), no. 1, 67–85. 

\smallskip
21. TSEMO, Aristide. Gerbes, 2-gerbes and symplectic fibrations. Rocky Mountain J. Math. 38 (2008), no. 3, 727 - 777. 
\end{document}